\documentclass[a4paper]{amsart}

\usepackage[english]{babel}
\usepackage{amsthm, latexsym, gastex, amssymb,url,listings}
\usepackage[mathcal]{eucal}
\usepackage{graphicx,url}

\copyrightinfo{2016}{Simone Ugolini}

\renewcommand{\phi}[0]{\varphi}
\renewcommand{\theta}[0]{\vartheta}
\renewcommand{\epsilon}[0]{\varepsilon}
\newcommand{\N}{\text{$\mathbf{N}$}}
\newcommand{\Z}{\text{$\mathbf{Z}$}}
\newcommand{\Q}{\text{$\mathbf{Q}$}}
\newcommand{\Pro}{\text{$\mathbb{P}^1$}}
\newcommand{\F}{\text{$\mathbb{F}$}}

\newtheorem{theorem}{Theorem}[section]
\newtheorem{lemma}[theorem]{Lemma}

\theoremstyle{definition}

\newtheorem{example}[theorem]{Example}
\theoremstyle{remark}
\newtheorem{remark}[theorem]{Remark}

\numberwithin{equation}{section}

\begin{document}
\bibliographystyle{amsplain}

\date{}

\title[]
{On the construction of irreducible polynomials over finite fields via odd prime degree endomorphisms of elliptic curves}

\author{S.~Ugolini}
\email{s.ugolini@unitn.it} 

\begin{abstract}
In this paper we present an iterative construction of irreducible polynomials  over finite fields based upon repeated applications of transforms induced by endomorphisms of odd prime degree of ordinary elliptic curves.  
\end{abstract}

\maketitle

\section{Introduction}
Several iterative constructions of irreducible polynomials with coefficients over finite fields have been proposed over the last decades. A comprehensive survey of such constructions can be found in \cite[Section 3.2]{hff}. 

Among the pioneering articles dealing with these constructions we would like to mention \cite{coh} and \cite{mey}, where the authors dealt with the so-called $Q$- and $R$-transform. Starting from an irreducible polynomial $f_0$ satisfying certain conditions on the coefficients and applying repeatedly one of these transforms it is possible to construct an infinite sequence of irreducible polynomials $\{ f_i \}_i$ such that 
\begin{equation*}
\deg(f_{i+1}) = 2 \cdot \deg(f_i)
\end{equation*}
for any $i$. In two  papers \cite{seq2, seqq} we showed how the constructions proposed by Meyn and Cohen behave when we remove any additional condition on the initial polynomial of the sequence, while in \cite{seqe} we  proposed a generalization of the $Q$- and $R$-transforms which is related to certain degree-$2$ endomorphisms of elliptic curves.

In the papers so far mentioned the degree of any polynomial of the sequence is (at most) twice the degree of the previous polynomial. The rationale behind the current paper is to rely on odd prime degree endomorphisms of elliptic curves to produce sequences of irreducible polynomials $\{ f_i \}_i$ over finite fields such that 
\begin{equation*}
\deg(f_{i+1}) = l \cdot \deg(f_i)
\end{equation*}
for a fixed odd prime $l$, at least for sufficiently large indices $i$. The iterative procedure introduced in Section \ref{sec_it_proc} is based upon the $r$-transform defined at the beginning of the same section. We notice that the synthesis of the first terms of the sequence might require some factorizations of a polynomial into its irreducible factors. By the way, as explained in Section \ref{sec_ana}, starting from an index $j$ it is possible to generate inexpensively all the polynomials of the sequence just setting $f_{i+1} := f_{i}^{r}$ for any $i \geq j$. 
\section{Background}\label{sec_bac}
Let $E$ be an ordinary elliptic curve defined over a finite field $\F_{p^n}$, where $p$ is an odd prime and $n$ a positive integer. We denote by $\pi_{{p^n}}$ the $\pi_{p^n}$-Frobenius endomorphism of $E$ over $\F_{p^n}$ and by $R$ its endomorphism ring, which we suppose to be a maximal order of a quadratic imaginary field $K = \Q (\sqrt{D})$ for some negative integer $D$.

We notice in passing that the requirement on the endomorphism ring of $E$ is not too restrictive. Indeed, one can construct ordinary elliptic curves over finite fields by reduction modulo a prime of elliptic curves defined over $\Q$. According to \cite[Remark 11.3.2]{sil}, up to isomorphism over $\overline{\Q}$ there are exactly $13$ elliptic curves over $\Q$ having complex multiplication. Moreover, $9$ of such elliptic curves have complex multiplication in the full ring of integers $\mathcal{O}_K$, where $K = \Q (\sqrt{D})$ for some $D$ in the set
\begin{equation*}
\{-1, -2, -3, -7, -11, -19, -43, -67, -163 \}.
\end{equation*}

We denote by $\alpha(x,y) = \left(r(x), y \cdot s (x) \right)$  an endomorphism of $E$ over $\F_{{p^n}}$ having prime degree $l$ coprime to the discriminant of $K$, in such a way that the kernel of $\alpha$ contains points of $E(\F_{p^n})$, where $r(x)$ and $s(x)$ are two rational functions in $\F_{p^n} (x)$ and
\begin{equation*}
r (x) = \frac{a(x)}{b(x)} 
\end{equation*}   
for certain polynomials $a(x)$ and $b(x)$ in $\F_{p^n} [x]$ such that $a(x)$ is monic. 

In the following we will denote by $\alpha$ also the representation of the endomorphism $\alpha$ in $R$. 

Let $d$ be a positive integer.

We introduce the following notations:
\begin{itemize}
\item $q:=p^{2nd}$;
\item $\rho:=\pi_{p^n}^{2d}$;
\item $\delta_i := \rho^{l^i}$, for any $i \in \N$.
\end{itemize}

In accordance with \cite[Theorem 1]{len}  we have that 
\begin{equation*}
E(\F_{q^{l^i}}) \cong R /(\delta_i-1) R
\end{equation*}
for any $i \in \N$.
Therefore we can study the action of $r$ on the points of $\Pro(\F_{q})$ by means of the action of $\alpha$ on the points in  
\begin{equation*}
R / (\delta_0-1) R.
\end{equation*}

For any $i \in \N$ we can find an element $\gamma_i$ in $R$ coprime to $\alpha$ and a positive integer $k_i$ such that
\begin{equation*}
\delta_i - 1 = \alpha^{k_i} \cdot \gamma_i.
\end{equation*} 
We define 
\begin{equation*}
\nu_{\alpha} (\delta_i-1) := k_i,
\end{equation*}
namely we denote by $\nu_{\alpha} (\delta_i-1)$ the exponent of the greatest power of $\alpha$ which divides $\delta_i-1$. 
 
We have that
\begin{equation*}
R / (\delta_i-1) R  \cong S_i,
\end{equation*}
where 
\begin{equation*}
S_i := R / \alpha^{k_i} R \times R / \gamma_i R.
\end{equation*}

In \cite{ugen} we investigated the structure of the graphs associated with the iterations of rational maps induced by endomorphisms of ordinary elliptic curves defined over finite fields. The graph associated with $r$ over $\F_q$ can be constructed as follows: we label the vertices by the elements of $\Pro (\F_q) = \F_q \cup \{ \infty \}$ and join a vertex $x_1$ with a vertex $x_2$ if $x_2 = r(x_1)$. In every connected component of the resulting graph there is exactly one cycle and a tree is attached to any vertex of the cycle.  

The following holds for the structure of the trees rising on the cycles of the graph $G^{q}_r$ associated with the dynamics of $r$ over $\Pro(\F_q)$.

\begin{theorem}\label{tree_depth}
Let $T^{q}_r (\tilde{x})$ be a tree of $G^{q}_r$ rooted in $\tilde{x} \in \Pro (\F_{p^{nd}})$. Then $T^{q}_r (\tilde{x})$ has depth $k_0$ and any leaf of the tree has height $k_0$.
\end{theorem}
\begin{proof}
First we notice that any vertex $x' \in T^{q}_r (\tilde{x})$ has height at most $k_0$. In fact, we notice that  any $r$-periodic element in $G^q_r$ is the $x$-coordinate of a point of the form $(0,P_1)$ in $S_0$. If we consider a point $P=(P_0,P_1)$ in $S_0$, then $[\alpha]^{k_0} (P) = (0, [\alpha]^{k_0} P_1)$, namely $[\alpha]^{k_0} (P)$ is a point of $R /({\delta_0}-1) R$ whose $x$-coordinate is $r$-periodic.  

Now consider a leaf $x' \in T^q_r (\tilde{x})$ and its corresponding point $P =([a], [b])$ in $S_0$. If $\alpha \mid a$, then $a = \alpha \cdot a'$ for some $a' \in R$. Moreover $[\alpha] (Q) = P$, where $Q=([a'],[a]^{-1} [b])$ and the $x$-coordinate $x_Q$ of $Q$ is such that $r(x_Q) = x'$ in contradiction with the fact that $x'$ is a leaf of $T^q_r(\tilde{x})$. Hence we deduce that $a$ is not divisible by $\alpha$ and the smallest integer $m$ such that $[\alpha]^m [a] = 0$ is $k_0$, namely $x'$ has height $k_0$.
\end{proof}

The following holds.
\begin{lemma}\label{val_lem}
Let $e$ be an integer such that $1 \leq e \leq l$. 

If $\eta = \delta_i$ for some $i \in \N$, then
\begin{displaymath}
\nu_{\alpha} (\eta^e-1) =
\begin{cases}
\nu_{\alpha} (\eta-1) & \text{ if $e < l$,}\\
\nu_{\alpha} (\eta-1) +1 & \text{ if $e = l$.}
\end{cases} .
\end{displaymath}
\end{lemma}
\begin{proof}
Whichever the value of $e$ is, we have that
\begin{equation}
\eta^e - 1 = (\eta -1) (\eta^{e-1}+ \dots + \eta +1).
\end{equation}

If $e < l$, then 
\begin{eqnarray*}
\eta^{e-1}+ \dots + \eta +1 & = & \sum_{j=1}^{e-1} (\eta^j-1) + e.
\end{eqnarray*}
Since $\nu_{\alpha} (\eta^j-1) \geq 1$ for any $j$, while $\nu_{\alpha} (e) = 0$, we conclude that 
\begin{equation*}
\nu_{\alpha} (\eta^e-1) = \nu_{\alpha} (\eta-1).
\end{equation*}

Now suppose that $e = l$. Then

\begin{eqnarray*}
\eta^{l-1}+ \dots + \eta +1 & = & \sum_{j=1}^{l-1} (\eta^j-1) + l \\
& = & (\eta-1) \cdot \sum_{j=1}^{l-1} \left( \sum_{k=0}^{j-1} \eta^k \right) + l\\
& = & (\eta-1) \cdot \sum_{j=1}^{l-1} \left( \sum_{k=0}^{j-1} (\eta^k-1) + j \right) + l\\
& = & (\eta-1) \cdot \left[ \sum_{j=1}^{l-1} \left( \sum_{k=0}^{j-1} (\eta^k-1) \right) + l \cdot \frac{l-1}{2} \right] + l.\\
\end{eqnarray*}
Since
\begin{displaymath}
\nu_{\alpha} (\eta-1) \geq 1 \quad \text{ and } \quad \nu_{\alpha} \left[ \displaystyle\sum_{j=1}^{l-1} \left( \displaystyle\sum_{k=0}^{j-1} (\eta^k-1) \right) + l \cdot \frac{l-1}{2} \right] \geq 1,
\end{displaymath}
we have that 
\begin{equation*}
\nu_{\alpha} \left( (\eta-1) \cdot \left[ \sum_{j=1}^{l-1} \left( \sum_{k=0}^{j-1} (\eta^k-1) \right) + l \cdot \frac{l-1}{2} \right] \right) \geq 2.
\end{equation*}
Hence
\begin{equation*}
\nu_{\alpha} (\eta^{l-1}+ \dots + \eta +1) = 1,
\end{equation*}
because $\nu_{\alpha} (l) = 1$.

Therefore $\nu_{\alpha} (\eta^l-1) = \nu_{\alpha} (\eta-1) + 1$.
\end{proof}

\section{The iterative construction of irreducible polynomials}\label{sec_it_proc}
For any polynomial $g(x)$ of positive degree defined over $\F_{{p^n}}$ we can consider its transform
\begin{equation*}
g^r (x) := (b(x))^{\deg(g)} \cdot g(r(x)). 
\end{equation*}

\subsection{The iterative procedure}
Consider a monic irreducible polynomial $f \in \F_{{p^n}} [x]$ of positive degree $d$. 

We set $i:=0$, $f_0:=f$ and enter Sub-procedure 1.

\textbf{Sub-procedure 1.}
\begin{itemize} 
\item \emph{Step 1:} set $f_{i+1}$ equal to one of the monic irreducible factors of $f_i^r$ in $\F_{{p^n}} [x]$.
\item \emph{Step 2:} if $f_{i+1}$ has degree greater than $2d$, then set $i:=i+1$ and go to Sub-procedure 2. Otherwise set $i:=i+1$ and go to Step 1.
\end{itemize}

\textbf{Sub-procedure 2.}
\begin{itemize} 
\item \emph{Step 1:} set $f_{i+1} := f_i^r$.
\item \emph{Step 2:} set $i:=i+1$ and go to Step 1.
\end{itemize}

\subsection{Analysis of the iterative procedure}\label{sec_ana}
Let $x_0 \in \F_{p^{dn}}$ be a root of $f_0$. Then there exists $y_0 \in \F_{q}$ such that $(x_0, y_0) \in E(\F_{q})$. Therefore $x_0$ belongs to the level $\tilde{j} \leq k_0$ of a tree $T_r^q (\tilde{x})$ rooted in an element $\tilde{x}$ (possibly $\tilde{x} = x_0$) belonging to $\Pro(\F_q)$ of the graph $G_r^q$. Moreover, according to Theorem \ref{tree_depth}, such a tree has depth $k_0$. 

By construction, one of the predecessors of $x_0$ in $G_r^q$ is a root of $f_1$.  

We consider separately two cases.
\begin{itemize}
\item \emph{Case $1$:} one predecessor of $x_0$ in $T_r^q (\tilde{x})$ is a root of $f_1$. Let $T(\tilde{x})$ be the (infinite) tree growing on $\tilde{x}$ and having vertices in $\Pro (\overline{\F_q})$. By construction one of the direct predecessors of $x_i$ in $T(\tilde{x})$ is a root of $f_{i+1}$ for any $i \geq 0$. Hence, for any $i \geq 0$, we define $x_{i+1}$ equal to such a root.

For any $i \leq k - \tilde{j}$ we have that $f_i$ has degree $d$ or $2d$ because $x_i$ is the $x$-coordinate of a rational point in $E(\F_q)$.

Let $j := k_0 - \tilde{j}$. Then $x_{j}$ is a leaf of $T_r^q (\tilde{x})$ and $f_j$ has degree $\tilde{d}$, where $\tilde{d} \in \{d, 2d \}$, while $x_{j+1}$ is the $x$-coordinate of a rational point in $E(\F_{q^e})$ for some positive integer $e$. As a consequence of Lemma \ref{val_lem} we can say that $e = l$. Indeed, $\nu_{\alpha} (\delta_0^e - 1) = \nu_{\alpha} (\delta_0-1)$ for any $e < l$.

More in general, for any integer $h \geq 0$ we have that $f_{j+h}$ is irreducible of degree $\tilde{d} l^{h}$ in $\F_{{p^n}} [x]$.

\item \emph{Case $2$:} no predecessor of $x_0$ in $T_r^q (\tilde{x})$ is a root of $f_1$. This happens when the only predecessor of $x_0$ in $G_r^q$ which is a root of $f_1$ is $r$-periodic. Since $T_r^q (\tilde{x})$ has depth $k_0$, we can stop Sub-procedure $1$ in case that $f_{k_0+1}$ has not degree greater than $2d$ and change the definition of $f_1$.
\end{itemize}

\begin{remark} In Sub-procedure 1 some polynomial factorizations might be required. More precisely, if $x_1$ is not $r$-periodic, then we need to factor no more than $k_0$ times a polynomial of degree at most $2d$ into its irreducible factors since $T^q_r (\tilde{x})$ has depth $k_0$. 
\end{remark}

\begin{example}
Consider the elliptic curve
\begin{equation*}
E : y^2 = x^3 + 56 x +34
\end{equation*}
defined over $\F_{83}$. This curve is the reduction modulo $83$ of the elliptic curve having Cremona label 5776g1 and has complex multiplication in $\Z \left[ \frac{1+ \sqrt{-19}}{2} \right]$. By means of a computational tool as Sage \cite{sage} we can check that such a curve is ordinary over $\F_{83}$. Using Sage we can also check that an endomorphism $\alpha$ of degree $17$ is defined on $E$ over $\F_{83}$, namely the endomorphism 
\begin{equation*}
\alpha (x,y) = (r(x), y \cdot s (x))
\end{equation*}
where $r(x) = \frac{a(x)}{b(x)}$ with

\begin{eqnarray*}
a(x) & = & 4 \cdot \left( x^{17} - 32 x^{16} + 13 x^{15} + 12 x^{14} - x^{13} - 32 x^{12} + 32 x^{11} - 10 x^{10} + 8 x^9 \right. \\
& & + 2 x^8 - 8 x^7 - x^6 +x^5 + 35 x^4 + 41 x^3 + 11 x^2 +29 x +22);\\
b(x) & = & \left( x^{16} - 32 x^{15} - 39 x^{14} + 19 x^{13} + 36 x^{12} + 8 x^{11} + 41 x^{10} - 8 x^{9} - 32  x^8 \right. \\
&  & - 16 x^7 - 41 x^6 - 13 x^5 - 3 x^4 + 5 x^3 + 3 x^2 - 10x +23).
\end{eqnarray*}
 
If we define $q:=83^3$ we can check that $q \not \equiv 1 \pmod{17}$. Therefore $q^{17^i} \not \equiv 1 \pmod{17}$ for all positive integers $i$. Hence, whichever monic irreducible polynomial of degree $3$ over $\F_{83}$ we take, it is possible to construct an infinite sequence of monic irreducible polynomials $\{f_i \}_i$ where the degree of any $f_i$ is divisible by $17^k$ for some positive integer $k$. For example, define 
\begin{equation*}
f_0:= x^3 + 3 x -2.
\end{equation*} 
We notice in passing that $f_0$ is irreducible over $\F_{83}$ since it is the Conway polynomial of degree $3$ over $\F_{83}$. 

Using GAP \cite{gap} we can produce an infinite sequence of irreducible polynomials as follows.

After defining $a$ and $b$ as above, we set 
\begin{lstlisting}
f:=ConwayPolynomial(83,3);
\end{lstlisting}

Then we set 
\begin{lstlisting}
g:=f;
g:=b^Degree(g)*Value(g,r);;g:=Factors(g)[1];;Degree(g);
\end{lstlisting}
\end{example}
If we iterate the command  
\begin{lstlisting}
g:=b^Degree(g)*Value(g,r);;g:=Factors(g)[1];;Degree(g);
\end{lstlisting}
we  notice that the degree of $g$ is always $3$. This is due to the fact that, in the case the irreducible polynomials $g$ we are going to construct have roots which are periodic with respect to the map $r$, the degree of the polynomials never increases. To fix up the issue we must re-initialize 
\begin{lstlisting}
g:=f;
\end{lstlisting}
and then set
\begin{lstlisting}
g:=b^Degree(g)*Value(g,r);;g:=Factors(g)[2];;Degree(g);
\end{lstlisting}
In so doing we get that the polynomial $g$ has degree $6$. We define such a polynomial as $f_1$. Then, if we set
\begin{lstlisting}
i:=1;
g:=b^Degree(f_i)*Value(f_i,r);;
i:=i+1;;
f_i:=Factors(g)[1];;
Degree(f_i);
\end{lstlisting}
we notice that the degree of $f_2$ is $102$, namely $17 \cdot 6$. From now on there is no need of further factorizations and at every iteration of the commands
\begin{lstlisting}
g:=b^Degree(f_i)*Value(f_i,r);;
i:=i+1;;
f_i:=g;;
\end{lstlisting}
we get an irreducible polynomial $f_i$ of degree $6 \cdot 17^{i-1}$. 

\bibliography{Refs}

\end{document}